\documentclass[11pt,a4paper]{amsart}
\usepackage{graphicx,multirow,array,amsmath,amssymb,enumitem}

\newtheorem{theorem}{Theorem}

\newtheorem{lemma}[theorem]{Lemma}

\begin{document}

\title{Lattice stick number of spatial graphs}

\author[H. Yoo]{Hyungkee Yoo}
\address{Department of Mathematics, Korea University, Seoul 02841, Korea}
\email{lpyhk727@korea.ac.kr}
\author[C. Lee]{Chaeryn Lee}
\address{Department of Mathematics, University of Illinois Urbana Champaign, Urbana, IL 61801}
\email{chaeryn2@illinois.edu}
\author[S. Oh]{Seungsang Oh}
\address{Department of Mathematics, Korea University, Seoul 02841, Korea}
\email{seungsang@korea.ac.kr}

\thanks{The corresponding author(Seungsang Oh) was supported by the National Research Foundation of Korea(NRF) grant funded by the Korea government(MSIP) (No. NRF-2017R1A2B2007216).}

\maketitle

\begin{abstract}
The lattice stick number of knots is defined to be the minimal number of straight sticks in the cubic lattice
required to construct a lattice stick presentation of the knot.
We similarly define the lattice stick number~$s_{L}(G)$ of spatial graphs $G$
with vertices of degree at most six (necessary for embedding into the cubic lattice),
and present an upper bound in terms of the crossing number~$c(G)$
$$ s_{L}(G) \leq 3c(G)+6e-4v-2s+3b+k, $$
where $G$ has $e$~edges, $v$~vertices, $s$~cut-components, $b$~bouquet cut-components,
and $k$~knot components.
\end{abstract}

\section{Introduction} \label{sec:intro}
All definitions and statements throughout this paper will concern the piecewise linear category.
A graph consists of vertices connected by edges allowing loops and multiple edges. 
A spatial graph is an embedding of a graph into $\mathbb{R}^3$.
Two spatial graphs are considered the same if there exists an ambient isotopy taking one to the other. 

An interesting question involving spatial graphs may ask 
how we can make a given spatial graph out of straight sticks.
A stick presentation of a spatial graph consists of straight sticks glued together end to end.
The minimal number of sticks necessary to construct a stick presentation of a spatial graph $G$ is 
defined as the {\em stick number} $s(G)$ of $G$.
Specially, we may choose to use sticks in the cubic lattice $\mathbb{Z}^3$.
A spatial graph embedded into $\mathbb{Z}^{3}$ is called its lattice stick presentation.
The length of each stick is a positive integer.
The {\em lattice stick number\/} $s_{L} (G)$ of $G$ is the minimal number of sticks necessary to construct 
a lattice stick presentation of $G$.
See Figure~\ref{fig:latticestick} for an example of a stick and lattice stick presentation of
the $\theta$-curve $3_1$.

\begin{figure}[h]
\includegraphics{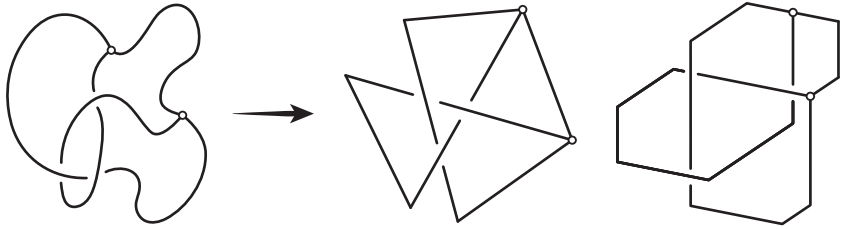}
\caption{Stick and lattice stick presentation}
\label{fig:latticestick}
\end{figure}

Much research has been made on the stick number and lattice stick number of knots and links.
As for the stick number of knots, 
Randell~\cite{R} proved that $s(3_1)=6$, $s(4_1)=7$, $s(K)=8$ for all prime knots with five or six crossings 
and $s(K) \geq 8$ for other nontrivial knots $K$.
The stick number of a $(p, q)$-torus knot with $p < q < 2p$ was shown to be $2q$ by Jin~\cite{Ji}.
Adams et al.~\cite{ACCJSZ} independently found the stick number of all $(n, n \! + \! 1)$-torus knots 
and their compositions.
Negami~\cite{Ne} found bounds on the stick number for nontrivial knots $K$ 
with respect to the crossing number $c(K)$ which is $\frac{5+ \sqrt{25+8(c(K)-2)}}{2} \leq s(K) \leq 2c(K)$.
The upper bound was improved by Huh and Oh~\cite{HO3} to $s(K) \leq \frac{3}{2} c(K)+\frac{3}{2}$.
McCabe~\cite{Mc} proved that $s(K) \leq c(K) +3$ for a $2$-bridge knot or link $K$ 
except the unlink and the Hopf link.
Huh, No and Oh~\cite{HNO} reduced this bound by 1 for a $2$-bridge knot or link $K$ with $c(K) \geq 6$.
The lower bound was improved by Calvo~\cite{Ca} to $\frac{7 + \sqrt{1+8c(K)}}{2} \leq s(K)$.      

There are also various results concerning the lattice stick number of knots and links.
It was proved by Diao~\cite{Di} that $s_{L}(3_1) = 12$.
Huh and Oh~\cite{HO1, HO2} showed that $s_{L}(4_1) = 14$ and
that all other nontrivial knots have lattice stick number at least 15.
Adams et al.~\cite{ACCJSZ} found that $s_{L}(T_{p,p+1}) = 6p$ for $p \geq 2$ where $T_{p,p+1}$ is 
a $(p, p + 1)$-torus knot.
All links with more than one component with lattice stick number at most 14 
were found by Hong, No and Oh~\cite{HNO2}.
Janse Van Rensburg and Promislow~\cite{JaP} proved that $s_{L}(K) \geq 12$ for nontrivial knots $K$.
Adams et al.~\cite{ACCJSZ} found an upper bound $s_{L}(K) \leq 6c(K)-4$.
An improved upper bound $s_{L}(K) \leq 3c(K)+2$ was found by Hong, No and Oh~\cite{HNO1} 
using the arc presentation of $K$.

We extend the study of the stick number and lattice stick number of knots to spatial graphs.
Lee, No and Oh~\cite{LNO2} found an upper bound of the stick number of spatial graphs $G$ which is
$s(G) \leq \frac{3}{2} c(G) + 2e - \frac{v}{2} + \frac{3b}{2}$
where $e$, $v$ and $b$ are the number of edges, vertices and bouquet cut-components, respectively.
The notion of bouquet cut-components will be explained just before the main theorem.

In this paper, we find an upper bound of the lattice stick number of spatial graphs.
We only consider spatial graphs which have vertices of degree at most six,
due to the structure of the cubic lattice $\mathbb{Z}^3$.
Following convention, we assume that a spatial graph $G$ does not have vertices of degree one or two,
except when it has a disjoint circle component, called a {\em knot component\/}.
We say that a knot component consists of a vertex and an edge.

\begin{figure}[h]
\includegraphics{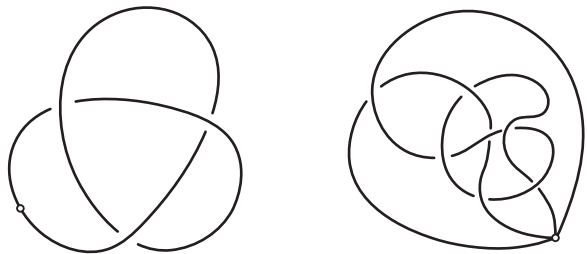}
\caption{Knot component and bouquet  cut-component}
\label{fig:bouquet}
\end{figure}

We consider two types of 2-spheres that separate $G$ into two parts.
Such a 2-sphere is called a {\em splitting-sphere\/} if it does not meet $G$,
and a {\em cut-sphere\/} if it intersects $G$ in a single vertex,
which is called a {\em cut-vertex\/}.
We maximally decompose $G$ into {\em cut-components\/} by cutting $G$ 
along a maximal set $\mathcal{S}$ of splitting-spheres and cut-spheres 
where any two spheres are either disjoint or intersect each other in a cut-vertex.
A cut-component of $G$ is called a {\em bouquet  cut-component\/} if it consists of one vertex and loops.
See Figure~\ref{fig:bouquet}.
Let ${\rm d}(x)$ denote the degree of a vertex $x$ of $G$.

\begin{theorem} \label{thm:bound}
Let $G$ be a spatial graph with  $3 \leq {\rm d}(x) \leq 6$ for each vertex $x$,
only allowing ${\rm d}(x) \! = \! 2$ for the unique vertex of each knot component.
If $G$ has $e$ edges, $v$ vertices, $s$ cut-components, $b$ bouquet cut-components
and $k$ knot components, then
$$ s_{L}(G) \leq 3c(G)+6e-4v-2s+3b+k. $$
\end{theorem}

It is noteworthy that for a knot $K$ this result implies $s_{L}(K) \leq 3c(K)+4$
which is close to the known upper bound $s_{L}(K) \leq 3c(K)+2$ in~\cite{HNO1}.

\section{Arc index of spatial graphs} \label{sec:arc}

There is an open book decomposition of $\mathbb{R}^3$ that has open
half-planes as pages and the standard $z$-axis as the binding axis.
We embed a spatial graph $G$ into an open book with finitely many pages
so that it meets each page in exactly one simple arc with two different end-points on the binding axis.
Therefore the binding axis contains all vertices of $G$ and
finitely many points from the interiors of the edges of $G$
and each edge of $G$ may pass from one page to another across the binding axis.
The realization of an arc presentation of this type for a spatial graph is introduced in~\cite{LNO1}.
An arc presentation of the $\theta$-curve $3_1$ with six pages is depicted in Figure~\ref{fig:arc}.
The {\em arc index\/} $\alpha(G)$ is defined to be the minimal number of
pages among all possible arc presentations of $G$.

\begin{figure}[h]
\includegraphics{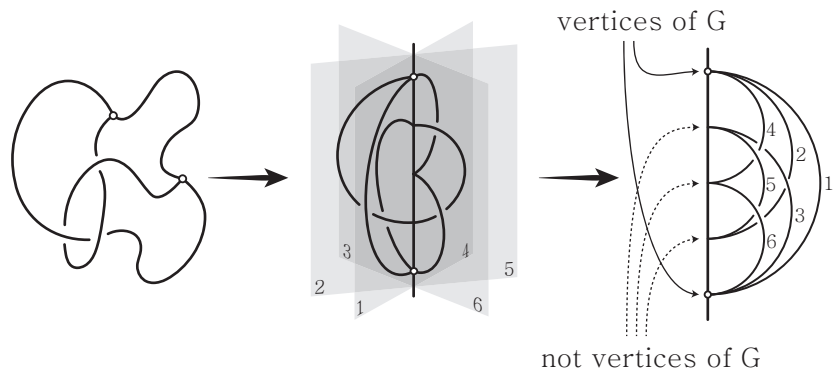}
\caption{Arc presentation}
\label{fig:arc}
\end{figure}

In an arc presentation of $G$,
we turn the pages so that they are all positioned to the right of the binding axis
and assign natural numbers $1, 2, \dots, \alpha(G)$ to all arcs from back to front
which are called the page numbers.
The points of $G$ on the binding axis are called the binding points.

\begin{lemma} \label{lem:binding}
Suppose that $G$ has $v$ vertices and $e$ edges.
If $\beta$ is the number of binding points of an arc presentation of $G$ with $\alpha(G)$ arcs then
$$ \beta = \alpha(G) +v -e.$$
\end{lemma}

\begin{proof}
Given an arc presentation of G with $\alpha(G)$ arcs, let us count the number of endpoints of all the arcs.
Note that each of the $\beta \! - \! v$ binding points which does not correspond to vertices of $G$ has degree 2.
$$ 2\alpha(G)=2(\beta-v)+\sum_{x \in V} {\rm d}(x), $$
where $V$ is the set of vertices of $G$.
Since $\sum_{x \in V}{\rm d}(x)=2e$, we have the desired equality.
\end{proof}

Bae and Park~\cite{BP} found an upper bound on the arc index of a nontrivial knot
in terms of the crossing number.
Lee, No and Oh generalized the argument used in Bae and Park's paper
to find an upper bound on the arc index of a spatial graph as follows.
This bound is crucial in the last part of the proof of the main theorem.

\begin{theorem} {\textup{(Lee-No-Oh~\cite{LNO1})}} \label{thm:LNO1}
Let $G$ be any spatial graph with $e$ edges and $b$ bouquet cut-components.
Then
$$ \alpha(G) \leq c(G)+e+b. $$
\end{theorem}

\section{Lattice presentation} \label{sec:stick}

In this section we prove Theorem~\ref{thm:bound}.
Let $G \! = \! (V,E)$ be a spatial graph with  $3 \leq {\rm d}(x) \leq 6$ for each vertex $x$,
only allowing ${\rm d}(x)=2$ for vertices of knot components.
Recall that each knot component consists of a vertex and an edge.
Suppose that $G$ has $e$ edges, $v$ vertices, $s$ cut-components,
$b$ bouquet cut-components and $k$ knot components.
Let $G=G_1 \cup G_2 \cup \dots \cup G_s$ be the cut-component decomposition of $G$.
Note that $\alpha(G)=\sum \alpha(G_k)$ as mentioned in~\cite[Claim~$2$]{LNO1}.

We may assume that $G$ is non-splittable,
hence no 2-sphere decomposing $G$ is a splitting-sphere.
First we focus on each cut-component.

\subsection{Lattice presentation of a cut-component} \label{subsec:component} \hspace{10mm}

A stick in the cubic lattice $\mathbb{Z}^{3}$ parallel to the $x$-axis is called an $x$-stick.
Each $yz$-plane on which the ends of an $x$-stick are placed is called an $x$-level.
Sticks and levels corresponding to the $y$- and $z$-coordinates are defined analogously.

Given an arc presentation of a spatial graph,
we move the binding axis to the line $y \! = \! x$ on the $xy$-plane
and replace each arc with two connected sticks,
one $x$-stick and one $y$-stick situated properly below the line $y \! = \! x$.
For a better view we slightly perturb some sticks if they overlap.
The result is called a {\em lattice arc diagram\/} of the spatial graph on the plane.
See Figure~\ref{fig:lattice}.

\begin{figure}[h]
\includegraphics{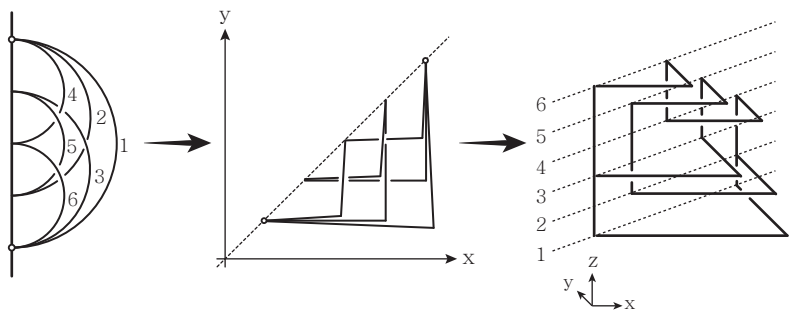}
\caption{Lattice arc diagram and lattice presentation}
\label{fig:lattice}
\end{figure}

We first construct a lattice presentation $L_k$ of a cut-component $G_k$ of $G$ in $\mathbb{Z}^{3}$.
Let $D_k$ be a lattice arc diagram of $G_k$ in the $xy$-plane with $\alpha(G_k)$ arcs
consisting of $\alpha(G_k)$ $x$-sticks and $\alpha(G_k)$ $y$-sticks.
To begin, place all arcs of $D_k$ on $z$-levels in the order of their page numbers starting from the bottom.
For each set of arcs that share a common binding point,
connect the arcs by adding successive $z$-sticks between the lowest and the highest $z$-levels.

We specially distinguish some classes of cut-components.
A cut-component of $G$ is called an {\em arc cut-component\/}
if it consists of one edge with two distinct end points
and a $\theta_n$ {\em cut-component\/} if it consists of two vertices 
and $n$ edges connecting them for $n \geq 2$.
A $\theta_n$ cut-component is called trivial if it can lie on a plane.

We are always able to remove two more sticks from $L_k$ as follows.
Consider the first binding point $b_1$ of $D_k$ where only $x$-sticks are attached.
Focus on the $x$- and $z$-sticks of $L_k$ on the $xz$-plane associated to $b_1$.
One can slide the $z$-sticks in the positive $x$-axis direction on this plane
until the shortest $x$-stick disappears as illustrated in Figure~\ref{fig:sliding}.
This means that we can reduce the number of $x$-sticks by one.
Similarly we repeat this replacement on the last binding point of $D_k$ as well to delete one $y$-stick.
Note that if $G_k$ is a trivial $\theta_n$ cut-component we can delete $n$ $x$-sticks
even though we slide on the $xz$-plane corresponding to binding point $b_1$ only;
if $G_k$ is an arc cut-component, we discuss the matter separately in Section~\ref{subsec:arc}.

\begin{figure}[h]
\includegraphics{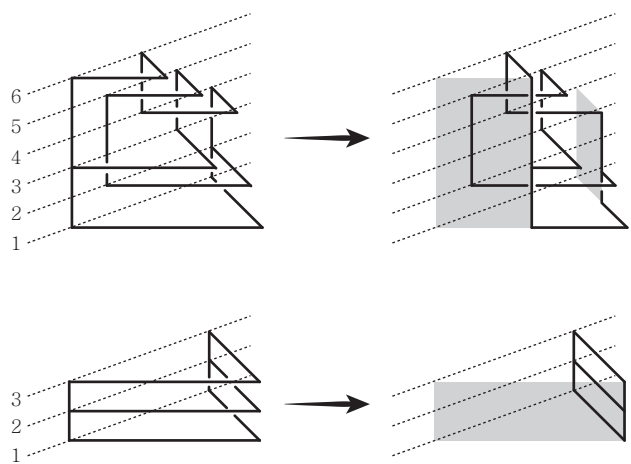}
\caption{Side-sliding}
\label{fig:sliding}
\end{figure}

This replacement is called {\em side-sliding\/}
and the resulting $L_k$ is our desired lattice presentation associated to $G_k$.
The reader must be aware that this lattice presentation $L_k$ is not exactly an embedding of $G_k$
because each vertex of $G_k$ with degree larger than 3 is replaced with
several vertices with degree 3 and connecting edges which are $z$-sticks.
We will adjust this problem in Section~\ref{subsec:merging} by vertex-merging.

\subsection{Combining all cut-components} \label{subsec:comb} \hspace{10mm}

Now we combine all lattice presentations associated to cut-components into one.
Even though there is a unique cut-component decomposition,
sets of decomposing cut-spheres are not unique.
Choose a set $\mathcal{S}$ of decomposing cut-spheres so that
given any two cut-spheres sharing a cut-vertex, one must be contained in the ball bounded by the other.

Since all cut-spheres in $\mathcal{S}$ are partially nested,
we construct a tree structure whose vertices are the cut-components $G_i$'s
and whose edges correspond to the related cut-spheres as in Figure~\ref{fig:tree}.
Re-order the indices of the $G_i$'s so that
$G_1$, as the root of the tree, indicates the cut-component lying in the unbounded region
and the remaining follow the partial order of the tree structure.

\begin{figure}[h]
\includegraphics{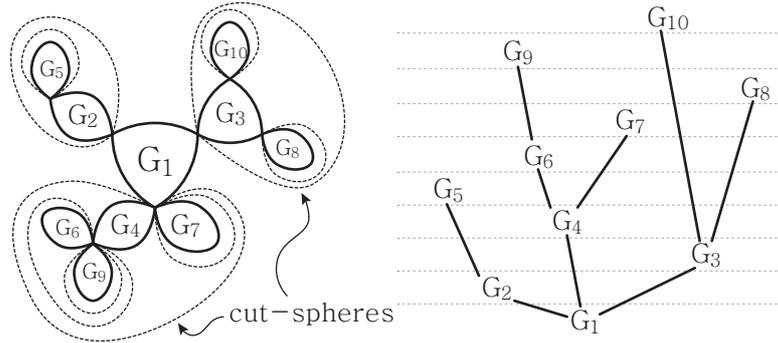}
\caption{Tree structure on the set of cut-components}
\label{fig:tree}
\end{figure}

For each $G_i$, its adjacent cut-components $G_j$ in the tree structure are called
{\em branch cut-components\/} of $G_i$ for $j \! > \! i$, and
{\em stem cut-component\/} otherwise as in Figure~\ref{fig:branch}.
Each cut-component except $G_1$ has a unique stem cut-component.
Furthermore, if two cut-components share a stem cut-component,
they cannot share a cut-vertex
since cut-spheres meeting at a cut-vertex are totally nested.

\begin{figure}[h]
\includegraphics{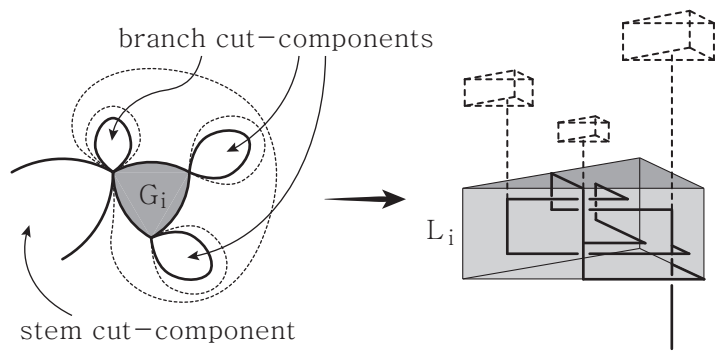}
\caption{Branch and stem cut-component}
\label{fig:branch}
\end{figure}

Now we stack up all the constructed lattice presentations $L_k$'s in inductive order with respect to the indices as follows.
First we put $L_1$ on the lowest floor.
After stacking $L_k$'s for $k \! < \! j$,
we connect $L_j$ to its unique stem cut-component $L_i$ ($i \! < \! j$).
There is a set of successive $z$-sticks of $L_i$ related to the cut-vertex of $G$ where $G_i$ and $G_j$ meet.
Let $\overline{p}_i$ be the topmost point of the union of these $z$-sticks.
Similarly let $\underline{p}_j$ be the bottommost point of the union of the successive $z$-sticks
of $L_j$ related to the same cut-vertex.
We take an infinite cylinder for $L_j$ whose core axis runs through $\overline{p}_i$ and is parallel to the $z$-axis.
Assume that the radius of the cylinder is sufficiently small.
Fit $L_j$ into the cylinder so that $\underline{p}_j$ lies on the core of the cylinder and
$L_j$ is positioned above $L_{j-1}$.
Connect the two points $\overline{p}_i$ and $\underline{p}_j$ by a $z$-stick as in Figure~\ref{fig:combing}.
Let $L$ be the lattice presentation obtained by combining all the $L_k$'s.
It is noteworthy that the word ``sufficiently small" requires
the set of constructed cylinders to be partially nested as in the tree structure.

\begin{figure}[h]
\includegraphics{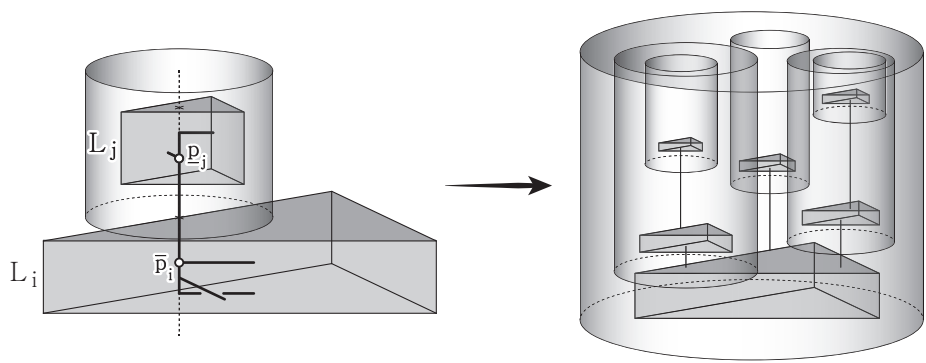}
\caption{Combining cut-components}
\label{fig:combing}
\end{figure}

\subsection{Vertex-merging} \label{subsec:merging} \hspace{10mm}

In this subsection we adjust this lattice presentation $L$ to be an embedding of $G$.
We only need to consider the arcs that are related to the vertices of $G$ with degree larger than 3.

Let $x$ be a vertex of $G$ with degree $d$ where $4 \leq d \leq 6$.
In the previous construction of $L$, vertex $x$ is replaced by $d \! - \! 2$ vertices with degree 3
as in the left picture depicted in Figure~\ref{fig:merging}.
There are $d \! - \! 1$ successive $z$-sticks and $d$ $x$- or $y$-sticks corresponding to $x$.
Assign $e_i$, $i \! = \! 1,\dots,d$, to these $x$- or $y$-sticks from bottom to top, 
and $z_i$ to the endpoint of $e_i$ corresponding to $x$.
In particular, $z_2$ is called the pivot vertex and will eventually indicate the vertex $x$.

\begin{figure}[h]
\includegraphics{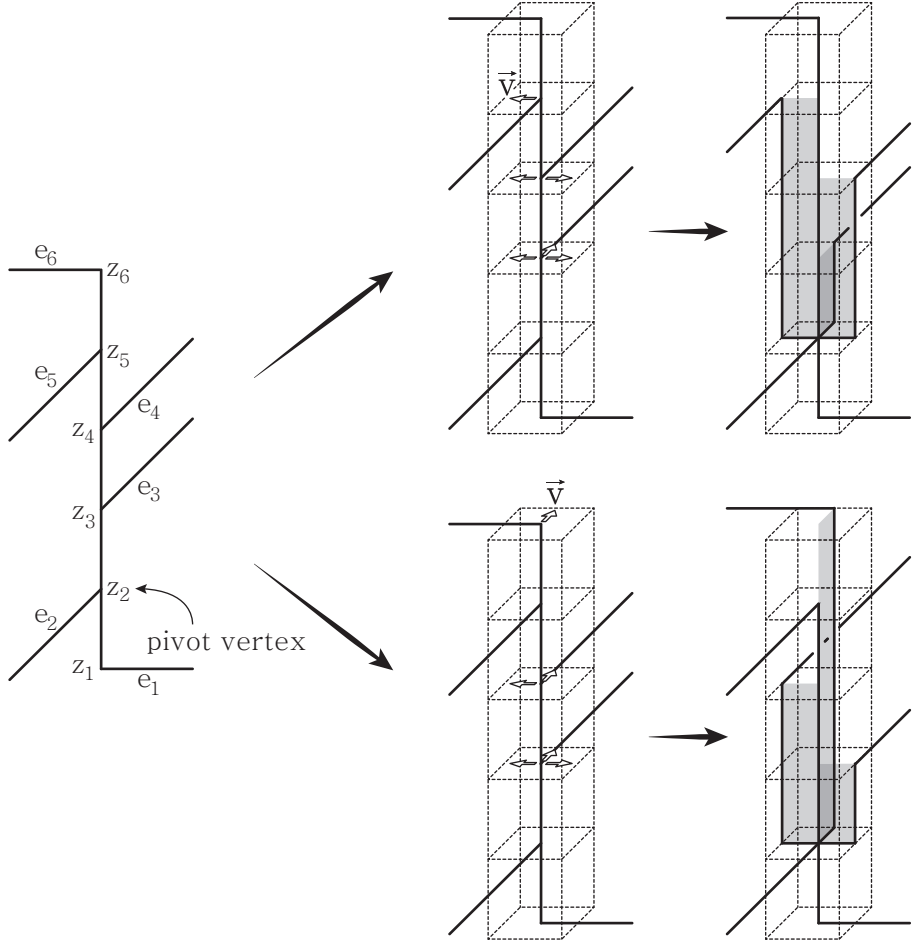}
\caption{Vertex-merging}
\label{fig:merging}
\end{figure}

We start by fixing the sticks $e_1$, $e_2$ and $e_d$.
Now we merge the vertex $z_3$ to the pivot vertex $z_2$ by sliding a small segment of $e_3$ near $z_3$
along the $z$-stick $z'$ connecting $z_2$ and $z_3$ as in Figure~\ref{fig:merging2}.
More precisely, we move this small segment of $e_3$ straight down along the shaded region until it meets $z_2$
and add a $z$-stick parallel to $z'$ connecting the broken endpoints.

There are still other ways to merge $z_3$ to $z_2$;
we slightly translate $e_3$ in one direction perpendicular to $e_3$ on the $xy$-plane.
Similarly, add a small $x$- or $y$-stick and a $z$-stick parallel to $z'$
connecting $z_2$ and the endpoint of $e_3$ near $z_3$.
Obviously the stick attached to the other endpoint of $e_3$ should be shrunk or extended
slightly to be connected to $e_3$.

This procedure is called {\em vertex-merging\/}.
During vertex-merging, vertex $z_3$ with degree 3 disappears,
pivot vertex $z_2$ becomes a vertex with degree 4
and the number of sticks increases by 1.
It is worth noting that there are at least two ways of applying vertex-merging to $z_3$ 
without newly merged $e_3$ colliding with $e_2$.

\begin{figure}[h]
\includegraphics{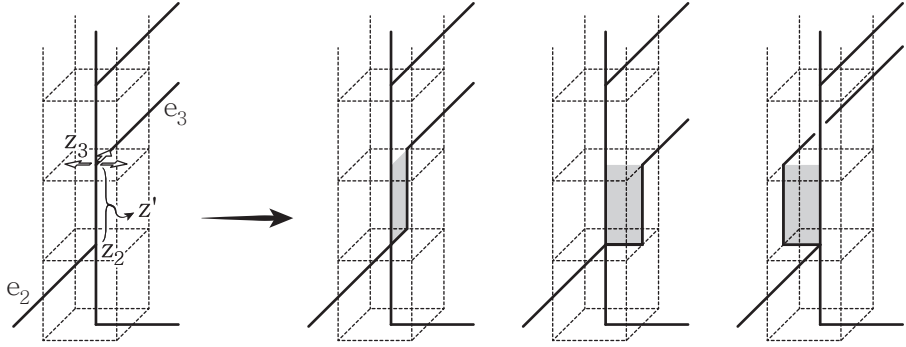}
\caption{Several ways of vertex-merging}
\label{fig:merging2}
\end{figure}

If $d \geq 5$, we similarly repeat vertex-merging on $z_4$
since there is still at least one way of applying vertex-merging to $z_4$ avoiding $e_2$
and the direction occupied by newly merged $e_3$ near $z_2$.

For $d=6$, there is a unique direction $\vec{v}$ which is not occupied by
$e_2$, newly merged $e_3$ nor $e_4$.
If the direction of $e_5$ is not opposite to $\vec{v}$,
then we perform vertex-merging on $z_5$ in the direction $\vec{v}$
as in the upper right picture in Figure~\ref{fig:merging}.
Otherwise, we fix $e_5$ and
instead perform vertex-merging on $z_6$ in the direction $\vec{v}$ as in the lower right picture.

Finally, the vertices $z_3, \dots, z_{d-1}$ with degree 3 are removed
and the pivot vertex $z_2$ becomes a vertex with degree $d$ to represent $x$ of $G$.
We repeat this procedure for all vertices of $G$ with degree larger than 3.
The resulting lattice presentation is now ambient isotopic to $G$.

We can always scale this presentation to fit into $\mathbb{Z}^{3}$
since there are only finitely many sticks which are parallel to $x$-, $y$- and $z$-axes.

\subsection{Reducing the number of sticks for arc cut-components} \label{subsec:arc} \hspace{10mm}

In Section~\ref{subsec:component},
two sticks are removed from each cut-component by side-sliding except from arc cut-components.
We now remove at least two sticks from each arc cut-component as well.

Let $G_k$ be an arc cut-component.
It has a unique stem cut-component $G_i$ and a unique branch cut-component $G_j$.
Even after vertex-merging,
we can take a cut-sphere $S$ separating $G_j$ from $G_k$ which intersects one endpoint $p$ of $G_k$.
We uniformly scale the ball bounded by $S$ small enough while fixing $p$.
This ball contains $G_j$ and all the branch cut-components of $G_j$.
We simply straighten the arc $G_k$ without touching the rest of $G$ as in Figure~\ref{fig:arccut}.
Obviously at least two sticks which are one $x$-stick and one $y$-stick
in the lattice presentation of $G_k$ are removed.

We again scale this presentation to fit into $\mathbb{Z}^{3}$.

\begin{figure}[h]
\includegraphics{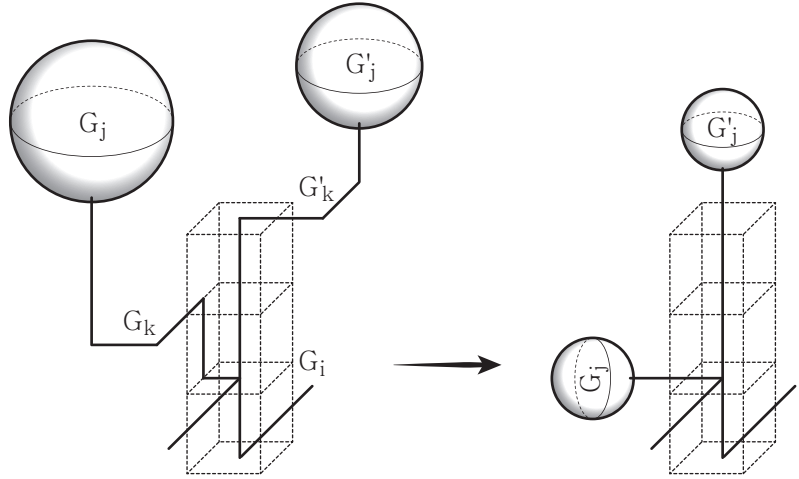}
\caption{Reducing the number of sticks for arc cut-components}
\label{fig:arccut}
\end{figure}

\subsection{Calculation} \label{subsec:cal} \hspace{10mm}

We now count the number of sticks necessary to construct the resulting lattice presentation $L$ of $G$.
Since one pair of $x$- and $y$-sticks comes from each arc of the lattice arc diagram,
$L$ initially had $\alpha(G)$ $x$-sticks and $\alpha(G)$ $y$-sticks.
As the result of side-sliding on each cut-component, we may reduce two more $x$- or $y$-sticks.
For each of the $\beta$ ($=\alpha(G)+v-e$ by Lemma~\ref{lem:binding}) binding points, we need at least one $z$-stick.
For each vertex $x$ with ${\rm d}(x) \! = \! 3$, one additional $z$-stick is necessary, and
for each vertex $x$ with ${\rm d}(x) \! \geq \! 4$,
we additionally need ${\rm d}(x) \! - \! 2$ $z$-sticks
and ${\rm d}(x) \! - \! 3$ $x$-sticks or $y$-sticks as the result of vertex-merging.
In summary,  the number of sticks of $L$ is
\begin{align*}
& 2\alpha(G)-2s+(\alpha(G)-e+v)+ \sum_{{\rm d}(x) = 3} 1 + \sum_{{\rm d}(x) \geq 4}(2{\rm d}(x)-5) \\
& = 3\alpha(G)-e+v-2s+\sum_{x \in V} (2{\rm d}(x)-5)+\sum_{{\rm d}(x) = 2} 1 \\
& = 3\alpha(G)-e+v-2s+2(2e)-5v+k \\
& = 3\alpha(G)+3e-4v-2s+k.
\end{align*}	

Finally, we get
\begin{align*}
s_{L}(G)	& \leq 3\alpha(G)+3e-4v-2s+k \\
	& \leq 3c(G)+6e-4v-2s+3b+k,
\end{align*}	
by applying Theorem~\ref{thm:LNO1}.
This completes the proof of Theorem~\ref{thm:bound}.

\end{document}